\input ./style/arxiv-vmsta.cfg
\documentclass[numbers,compress,v1.0.1]{vmsta}
\usepackage{vtexbibtags}

\usepackage{mathbh}

\volume{4}
\issue{4}
\pubyear{2017}
\firstpage{381}
\lastpage{406}
\aid{VMSTA91}
\doi{10.15559/17-VMSTA91}
\articletype{research-article}

\startlocaldefs
\newcommand{\lleft}{\left}
\newcommand{\rright}{\right}

\newcommand{\rrvert}{\vert}
\newcommand{\llvert}{\vert}
\allowdisplaybreaks

\newcommand{\e}{\mathbb{E}}
\newcommand{\law}{\stackrel{\text{law}}{=}}

\newcommand{\ph}{\phi^{(H)}}
\newcommand{\ha}{\hat{\gamma}_T}
\newcommand{\bo}{{\hat{\pmb{\gamma}}}}
\newcommand{\ii}{\mathbh{1}_{g\left(\bo_T\right)>0}}
\newcommand{\iii}{\mathbh{1}_{\ha(N) \neq0}}
\newtheorem{defi}{Definition}[]
\newtheorem{lemma}[]{Lemma}

\newtheorem{theorem}[]{Theorem}
\newtheorem{kor}[]{Corollary}

\newtheorem{rem}[]{Remark}
\newtheorem{ex}[]{Example}
\endlocaldefs

\begin{document}

\begin{frontmatter}
\pretitle{Research Article}

\title{On model fitting and estimation of strictly stationary processes}

\author[a]{\inits{M.}\fnms{Marko}~\snm{Voutilainen}\thanksref{cor1}\ead[label=e1]{marko.voutilainen@aalto.fi}}
\thankstext[type=corresp,id=cor1]{Corresponding author.}
\author[b]{\inits{L.}\fnms{Lauri}~\snm{Viitasaari}\ead[label=e2]{lauri.viitasaari@iki.fi}}
\author[a]{\inits{P.}\fnms{Pauliina}~\snm{Ilmonen}\ead[label=e3]{pauliina.ilmonen@aalto.fi}}

\address[a]{Department of Mathematics and Systems Analysis\\
\institution{Aalto University School of Science}\\
P.O. Box 11100, FI-00076 Aalto, \cny{Finland}}
\address[b]{Department of Mathematics and Statistics\\
University of Helsinki\\
P.O. Box 68, FI-00014 \institution{University of Helsinki}, \cny{Finland}}



\markboth{M. Voutilainen et al.}{On model fitting and estimation of strictly stationary processes}

\begin{abstract}
Stationary processes have been extensively studied in the literature.
Their applications
include modeling and forecasting numerous real life phenomena such as natural
disasters, sales and market movements. When stationary processes are
considered, modeling is traditionally based on fitting an
autoregressive moving average (ARMA) process. However, we challenge
this conventional approach. Instead of fitting an ARMA model, we apply
an AR(1) characterization in modeling any strictly stationary
processes. Moreover, we derive consistent and asymptotically normal
estimators of the corresponding model parameter.
\end{abstract}

\received{\sday{13} \smonth{9} \syear{2017}}
\revised{\sday{22} \smonth{11} \syear{2017}}
\accepted{\sday{25} \smonth{11} \syear{2017}}
\publishedonline{\sday{22} \smonth{12} \syear{2017}}
\end{frontmatter}

\section{Introduction}
Stochastic processes are widely used in modeling and forecasting
numerous real life phenomena such as natural disasters, activity of the
sun, sales of a company and market movements, to mention a few.
When stationary processes are considered, modeling is traditionally
based on fitting an autoregressive moving average (ARMA) process.
However, in this paper, we challenge this conventional approach.
Instead of fitting an ARMA model, we apply the AR$(1)$ characterization
in modeling any strictly stationary processes. Moreover, we derive
consistent and asymptotically normal estimators of the corresponding
model parameter.

One of the reasons why ARMA processes have been in a central role in
modeling of time-series data is that for every autocovariance function
$\gamma(\cdot)$ vanishing at infinity and for every $n\in\mathbb{N}$
there exists an ARMA process $X$ such that $\gamma(k) = \gamma_X(k)$
for every $k = 0,1,..,n$. For a general overview of the theory of
stationary ARMA processes and their estimation, the reader may consult
for example \citep{brockwell} or \citep{hamilton1994time}.

ARMA processes, and their extensions, have been studied extensively in
the literature.
A direct proof of consistency and asymptotic normality of Gaussian
maximum likelihood estimators for causal and invertible ARMA processes
was given in \citep{yao2006gaussian}. The result was originally
obtained, using asymptotic properties of the Whittle estimator, in
\citep{hannan1973asymptotic}. The estimation of the parameters of
strictly stationary ARMA processes with infinite variances was studied
in \citep{miko}, again, by using Whittle estimators. Portmanteau tests
for ARMA models with stable Paretian errors with infinite variance were
introduced in \citep{lin2008portmanteau}. An efficient method for
evaluating the maximum likelihood function of stationary vector ARMA
models was presented in \citep{mauricio1995exact}. Fractionally
integrated ARMA models with a GARCH noise process, where the variance
of the error terms is also of ARMA form, was studied in \citep
{ling1997fractionally}. Consistency and asymptotic normality of the
quasi-maximum likelihood estimators of ARMA models with the noise
process driven by a GARCH model was shown in \citep{francq2004maximum}.
A least squares approach for ARMA parameter estimation has been studied
at least in \citep{koreisha1990generalized} by contrasting its
efficiency with the maximum likelihood estimation.
Also estimators of autocovariance and their limiting behavior have been
addressed in numerous papers. See for example \citep{davis1986limit,
horvath2008sample, levy2011robust} and \citep
{mcelroy2012subsampling}.


Modeling an observed time-series with an ARMA process starts by fixing
the orders of the model. This is often done by an educated guess, but
there also exists methods for estimating the orders, see e.g. \citep
{hannan1980estimation}. After the orders are fixed, the related
parameters can be estimated, for example, by using the maximum
likelihood or least squares estimators. These estimators are expressed
in terms of optimization problems and do not generally admit closed
form representations. The final step is to conduct various diagnostic
tests to determine whether the estimated model is sufficiently good or
not. These tests are often designed to recognize whether the residuals
of the model support the underlying assumptions about the error terms.
Depending on whether one considers strict or weak stationarity, the
error process is usually assumed to be an IID process or white noise,
respectively. If the tests do not support the assumptions about the
noise process, then one has to start all over again. Tests for the
goodness of fit of ARMA models
have been suggested e.g. in \citep{francq2005diagnostic}. 


The approach taken in this paper is based on the discrete version of
the main theorem of \citep{asymptotics} leading to an AR$(1)$
characterization for (any) strictly stationary processes. Note that
this approach covers, but is not limited to, strictly stationary ARMA
processes. 
It was stated in \citep{asymptotics} that a process is strictly
stationary if and only if for every fixed $0<H<1$ it can be represented
in the AR$(1)$ form with $\phi= e^{-H}$ and a unique, possibly
correlated, noise term. Although the representation is unique only
after $H$ is fixed, we show that in most of the cases, given just one
value of the autocovariance function of the noise, one is able to
determine the AR$(1)$ parameter and, consequently, the entire
autocovariance function of the noise process. It is worth emphasizing
that since the parameter--noise pair in the AR$(1)$ characterization is
not unique, it is natural that some information about the noise has to
be assumed. Note that conventionally, when applying ARMA models,
we have assumptions about the noise process much stronger than being
IID or white noise.
That is, the autocovariance function of the noise is assumed to be
identically zero except at the origin. When founding estimation on the
AR$(1)$ characterization, one does not have to select between different
complicated models. In addition, there is only one parameter left to be
estimated. Yet another advantage over classical ARMA estimation is that
we obtain closed form expressions for the estimators.

The paper is organized as follows. We begin Section~\ref{model} by
introducing some terminology and notation. After that, we give a
characterization of discrete time strictly stationary processes as
AR$(1)$ processes with a possibly correlated noise term together with
some illustrative examples. The AR$(1)$ characterization leads to
Yule--Walker type equations for the AR$(1)$ parameter $\phi$. In this
case, due to the correlated noise process, the equations are of
quadratic form in $\phi$. For the rest of the section, we study the
quadratic equations and determine $\phi$ with as little information
about the noise process as possible.
The approach taken in Section~\ref{model} leads to an estimator of the
AR$(1)$ parameter. We consider estimation in detail in Section~\ref{estimation}.
The end of Section~\ref{estimation} is dedicated to
testing the assumptions
we make when constructing the estimators.
A simulation study to assess finite sample properties of the estimators
is presented in Section~\ref{simulations}.
Finally, we end the paper with three appendices containing a technical
proof, detailed discussion on some special cases and tabulated
simulation results.

\section{On AR$(1)$ characterization in modeling strictly stationary processes}
\label{model}

Throughout the paper we consider strictly stationary processes.
\begin{defi}
Assume that $X = (X_t)_{t\in\mathbb{Z}}$ is a stochastic process. If
\begin{equation*}
(X_{t+n_1}, X_{t+n_2},\ldots,X_{t+n_k} ) \overset{\text
{law}} {=} (X_{n_1}, X_{n_2},\ldots,X_{n_k} )
\end{equation*}
for all $k\in\mathbb{N}$ and $t, n_1, n_2,\ldots,n_k \in\mathbb{Z}$, then
$X$ is strictly stationary.
\end{defi}

\begin{defi}
Assume that $G = (G_t)_{t\in\mathbb{Z}}$ is a stochastic process and
denote $\Delta_{t} G = G_t - G_{t-1}$. If $ (\Delta_{t} G
)_{t\in\mathbb{Z}}$ is strictly stationary, then the process $G$ is a
strictly stationary increment process.
\end{defi}
The following class of stochastic processes was originally introduced
in \citep{asymptotics}.

\begin{defi}
\label{G}
Let $H > 0$ be fixed and let $G= (G_t)_{t\in\mathbb{Z}}$ be a
stochastic process. If $G$ is a strictly stationary increment process
with $G_0 = 0$ and if the limit
\begin{equation}
\label{limit} \lim_{k\to-\infty} \sum_{t=k}^0
e^{tH}\Delta_{t} G
\end{equation}
exists in probability and defines an almost surely finite random
variable, then $G$ belongs to the class of converging strictly
stationary increment processes, and we denote $G\in\mathcal{G}_H$.
\end{defi}
Next, we consider the AR$(1)$ characterization of strictly stationary
processes. The continuous time analogy was proved in \citep
{asymptotics} together with a sketch of a proof for the discrete case.
For the reader's convenience, a detailed proof of the discrete case is
presented in Appendix~\ref{techproof}.

\begin{theorem}
\label{main}
Let $H >0$ be fixed and let $X = (X_t)_{t\in\mathbb{Z}}$ be a
stochastic process. Then $X$ is strictly stationary if and only if $\lim_{t\to-\infty}e^{tH}X_t = 0$ in probability and
\begin{equation}
\label{langevin} \Delta_t X = \bigl(e^{-H} - 1
\bigr)X_{t-1} + \Delta_t G
\end{equation}
for a unique $G\in\mathcal{G}_H$.
\end{theorem}

\begin{kor}\label{cor:stat}
Let $H>0$ be fixed. Then every discrete time strictly stationary
process $(X_t)_{t\in\mathbb{Z}}$ can be represented as
\begin{equation}
\label{form} X_t - \phi^{(H)} X_{t-1} =
Z_t^{(H)},
\end{equation}
where $\phi^{(H)} = e^{-H}$ and $Z_t^{(H)} = \Delta_t G$ is another
strictly stationary process.
\end{kor}

It is worth to note that the noise $Z$ in Corollary \ref{cor:stat} is
unique only after the parameter $H$ is fixed. The message of this
result is that every strictly stationary process is an AR$(1)$ process
with a strictly stationary noise that may have a non-zero
autocovariance function. The following examples show how some
conventional ARMA processes can be represented as an AR$(1)$ process.
\begin{ex}
\label{exar}
Let $X$ be a strictly stationary AR$(1)$ process defined by
\begin{equation*}
X_t - \varphi X_{t-1} = \epsilon_t, \qquad(
\epsilon_t)\sim IID\bigl(0, \sigma^2\bigr)
\end{equation*}
with $\varphi>0$. Then we may simply choose $\phi^{(H)} = \varphi$ and
$Z_t^{(H)} = \epsilon_t$.
\end{ex}

\begin{ex}
\label{exarma}
Let $X$ be a strictly stationary ARMA$(1,q)$ process defined by
\begin{equation*}
X_t - \varphi X_{t-1} = \epsilon_t +
\theta_1\epsilon_{t-1}+\cdots +\theta_q
\epsilon_{t-q}, \qquad (\epsilon_t )\sim IID\bigl(0,
\sigma^2\bigr)
\end{equation*}
with $\varphi>0$. Then we may set $\phi^{(H)} = \varphi$, and
$Z_t^{(H)}$ then equals to the MA$(q)$ process.
\end{ex}

\begin{ex}
Consider a strictly stationary AR$(1)$ process $X$ with $\varphi<0$.
Then $X$ admits an MA$(\infty)$ representation
\begin{equation*}
X_t = \sum_{k=0}^\infty
\varphi^k \epsilon_{t-k}.
\end{equation*}
From this it follows that
\begin{equation*}
Z_t^{(H)} = \epsilon_t + \sum
_{k=0}^\infty\varphi^k \bigl(\varphi-\phi
^{(H)} \bigr)\epsilon_{t-1-k}
\end{equation*}

\noindent and
\begin{equation*}
\text{cov}\bigl(Z_t^{(H)}, Z_0^{(H)}
\bigr) = \varphi^{t-2} (\varphi-\ph )\sigma^2 \Biggl(\varphi+
\bigl(\varphi-\ph \bigr)\sum_{n=1}^\infty \bigl(
\varphi^2\bigr)^n \Biggr).
\end{equation*}
Hence in the case of an AR$(1)$ process with a negative parameter, the
autocovariance function of the noise $Z$ of the representation \eqref
{form} is non-zero everywhere.
\end{ex}

Next we show how to determine the AR$(1)$ parameter $\phi^{(H)}$ in
\eqref{form} provided that the observed process $X$ is known. In what
follows, we omit the superindices in \eqref{form}. We assume that the
second moments of the considered processes are finite and that the
processes are centered. That is, $\e(X_t) = \e(Z_t) = 0$ for every $t\in
\mathbb{Z}$. Throughout the rest of the paper, we use the notation
cov$(X_t, X_{t+n}) = \gamma(n)$ and cov$(Z_t,Z_{t+n}) = r(n)$ for every
$t,n\in\mathbb{Z}$.

\begin{lemma}
\label{lemma:equations}
Let centered $(X_t)_{t\in\mathbb{Z}}$ be of the form \eqref{form}. Then
\begin{equation}
\label{quadratic} \phi^2 \gamma(n) - \phi \bigl(\gamma(n+1) +
\gamma(n-1) \bigr) + \gamma (n) - r(n) = 0
\end{equation}
for every $n\in\mathbb{Z}$.
\end{lemma}
\begin{proof}
Let $n\in\mathbb{Z}$. By multiplying both sides of
\begin{equation*}
X_n - \phi X_{n-1} = Z_n
\end{equation*}
with $Z_0 = X_0 - \phi X_{-1}$ and taking expectations, we obtain
\begin{align*}
 &\e \bigl(X_n(X_0- \phi
X_{-1}) \bigr)- \phi\e \bigl(X_{n-1}(X_0 - \phi
X_{-1}) \bigr)
\\
&\quad=\ \phi^2\gamma(n) - \phi\bigl(\gamma(n+1) + \gamma(n-1)\bigr)+
\gamma(n) = r(n).\qedhere
\end{align*}
\end{proof}

\begin{kor}
\label{estimator2}
Let centered $(X_t)_{t\in\mathbb{Z}}$ be of the form \eqref{form} and
let $N\in\mathbb{N}$ be fixed.
\begin{itemize}
\item[(1)] If $\gamma(N) \neq0$, then either
\begin{equation}
\label{solution} \hspace*{-0.53cm} \phi\,{=}\, \frac{\gamma(N\,{+}\,1) \,{+}\, \gamma(N\,{-}\,1) \,{+}\, \sqrt{ (\gamma(N\,{+}\,1)\,{+}\,\gamma
(N\,{-}\,1) )^2 \,{-}\, 4\gamma(N)(\gamma(N) \,{-}\, r(N))}}{2\gamma(N)}
\end{equation}
or
\begin{equation}
\label{solution2} \hspace*{-0.63cm} \phi\,{=}\, \frac{\gamma(N\,{+}\,1) \,{+}\, \gamma(N\,{-}\,1) \,{-}\, \sqrt{ (\gamma(N\,{+}\,1)\,{+}\,\gamma
(N\,{-}\,1) )^2 \,{-}\, 4\gamma(N)(\gamma(N) \,{-}\, r(N))}}{2\gamma(N)}.
\end{equation}
\item[(2)] If $\gamma(N) = 0$ and $r(N) \neq0$, then
\begin{equation*}
\phi= -\frac{r(N)}{\gamma(N+1) + \gamma(N-1)}.
\end{equation*}
\end{itemize}
\end{kor}

\noindent Note that if $\gamma(N) = r(N) = 0$, then Lemma \ref
{lemma:equations} yields only $\gamma(N+1) + \gamma(N-1) = 0$ providing
no information about the parameter $\phi$. As such, in order to
determine the parameter $\phi$, we require that either $\gamma(N) \neq
0$ or $r(N) \neq0$.

\begin{rem}
If the variance $r(0)$ of the noise is known, then \eqref{solution} and
\eqref{solution2} reduces to
\begin{equation*}
\phi= \frac{\gamma(1) \pm\sqrt{\gamma(1)^2 - \gamma(0) (\gamma(0)
- r(0) )}}{\gamma(0)}.
\end{equation*}
\end{rem}

\noindent At first glimpse it seems that Corollary \ref{estimator2} is
not directly applicable. Indeed, in principle it seems like there could
be complex-valued solutions although representation \eqref{form}
together with \eqref{quadratic} implies that there exists a solution
$\phi\in(0,1)$. Furthermore, it is not clear whether the true value is
given by \eqref{solution} or \eqref{solution2}. We next address these
issues. We start by proving that the solutions to \eqref{quadratic}
cannot be complex. At the same time we are able to determine which one
of the solutions one should choose.

\begin{lemma}
\label{D}
The discriminants of \eqref{solution} and \eqref{solution2} are always
non-negative.
\end{lemma}
\begin{proof}
Let $k\in\mathbb{Z}$. By multiplying both sides of \eqref{form} with
$X_{t-k}$, taking expectations, and applying \eqref{form} repeatedly we obtain
\begin{align*}
\gamma(k) - \phi\gamma(k-1) &= \e(Z_tX_{t-k}) = \e
\bigl(Z_t(Z_{t-k}+\phi X_{t-k-1})\bigr)
\nonumber
\\
&= r(k) + \phi\e(Z_tX_{t-k-1} )
\nonumber
\\
&= r(k) + \phi\e\bigl(Z_t(Z_{t-k-1} + \phi X_{t-k-2})
\bigr)
\nonumber
\\
&=r(k) + \phi r(k+1) + \phi^2\e (Z_t X_{t-k-2}
).
\end{align*}
Proceeding as above $l$ times we get
\begin{equation*}
\gamma(k) - \phi\gamma(k-1) = \sum_{i=0}^{l-1}
\phi^i r(k+i) + \phi ^{l}\e\bigl(Z_t(\phi
X_{t-k-l-2})\bigr).
\end{equation*}
Letting $l$ approach infinity leads to
\begin{equation}
\gamma(k) - \phi\gamma(k-1) = \sum_{i=0}^\infty
\phi^i r(k+i)\label{sonta},
\end{equation}
where the series converges as $r(k+i)\leq r(0)$ and $0<\phi<1$. It now
follows from \eqref{sonta} that
\begin{equation*}
\begin{split} \gamma(N) &= \phi\gamma(N-1) + \sum
_{i=0}^\infty\phi^{i} r(N+i)
\\
&=\phi\gamma(N-1) +r(N)+\phi\sum_{i=1}^\infty
\phi^{i-1}r(N+i)
\\
&=\phi\gamma(N-1) +r(N)+\phi\sum_{i=0}^\infty
\phi^{i}r(N+i+1)
\\
&= \phi\gamma(N-1) + \phi \bigl(\gamma(N+1) - \phi\gamma(N) \bigr)+r(N).
\end{split} %
\end{equation*}
Denote the discrimant of \eqref{solution} and \eqref{solution2} by $D$.
That is,
\[
D = \bigl(\gamma(N-1)+\gamma(N+1) \bigr)^2 - 4\gamma(N) \bigl(\gamma
(N)-r(N) \bigr).
\]
By using the equation above we observe that
\begin{equation*}
\begin{split} D &= \biggl(\frac{\gamma(N) + \phi^2\gamma(N)-r(N)}{\phi} \biggr)^2 -
4\gamma(N) \bigl(\gamma(N)-r(N)\bigr). \end{split} %
\end{equation*}
Denoting $a_N = \frac{r(N)}{\gamma(N)}$, multiplying by $\frac{\phi
^2}{\gamma(N)^2}$, and using the identity
\[
(a+b)^2 -4ab = (a-b)^2
\]
yields
\begin{equation*}
\begin{split} \frac{\phi^2}{\gamma(N)^2}D = \bigl(1+\phi^2-a_N
\bigr)^2 - 4\phi^2(1-a_N) = \bigl(\phi
^2-1+a_N\bigr)^2 \geq0. \end{split}
\end{equation*}
This concludes the proof.
\end{proof}

Note that if $r(N) = 0$, as $\phi<1$, the discriminant is always
positive. Let $a_N = \frac{r(N)}{\gamma(N)}$. The proof above now gives
us the following identity
\begin{equation*}
\phi= \frac{1}{2\phi} \biggl(1+\phi^2-a_N\pm
\frac{|\gamma(N)|}{\gamma
(N)} \bigl\llvert \phi^2-1+a_N \bigr\rrvert
\biggr).
\end{equation*}
This enables us to consider the choice between \eqref{solution} and
\eqref{solution2}. Assume that $\gamma(N) > 0$. If $\phi^2 - 1 + a_N >
0$, then $\phi$ is given by \eqref{solution} (as $\phi\in(0,1)$).
Similarly, if $\phi^2 - 1 +a_N<0$, then $\phi$ is determined by \eqref
{solution2}. Finally, contrary conclusions hold in the case $\gamma(N)
< 0$. In particular, we can always choose between \eqref{solution} and
\eqref{solution2} provided that either $a_N \leq0$ or $a_N \geq1$.
Moreover, from \eqref{quadratic} it follows that
\begin{equation*}
\frac{r(N)}{\gamma(N)} = \frac{r(N+k)}{\gamma(N+k)}
\end{equation*}
if and only if
\begin{equation*}
\frac{\gamma(N+1)+\gamma(N-1)}{\gamma(N)} = \frac{\gamma(N+1+k)+\gamma
(N-1+k)}{\gamma(N+k)},
\end{equation*}
provided that the denominators differ from zero. Since \eqref{solution}
and \eqref{solution2} can be written as
\begin{align}
\phi&= \frac{\gamma(N+1)+\gamma(N-1)}{2\gamma(N)} \notag\\
&\quad\pm\frac{1}{2}
\text {sgn}\bigl(\gamma(N)\bigr)\sqrt{ \biggl(\frac{\gamma(N+1)+\gamma(N-1)}{\gamma
(N)}
\biggr)^2-4 \biggl(1-\frac{r(N)}{\gamma(N)} \biggr)},\label{bform}
\end{align}
we observe that one can always rule out one of the solutions \eqref
{solution} and \eqref{solution2} provided that $a_N \neq a_{N+k}$.
Therefore, it always suffices to know two values of the autocovariance
$r$ such that $a_N \neq a_{N+k}$, except the worst case scenario where
$a_j = a \in(0,1)$ for every $j\in\mathbb{Z}$. A detailed analysis of
this particular case is given in Appendix~\ref{worstcase}.

\begin{rem}
Consider a fixed strictly stationary process $X$. If we fix one value
of the autocovariance function of the noise such that Corollary \ref
{estimator2} yields an unambiguous AR$(1)$ parameter, then the
quadratic equations \eqref{quadratic} will unravel the entire
autocovariance function of the noise process. In comparison,
conventionally, the noise is assumed to be white --- meaning that the
entire autocovariance function of the noise is assumed to be known \emph
{a priori}.
\end{rem}

We end this section by observing that in the case of vanishing
autocovariance function of the noise, we get the following simplified
form for the AR$(1)$ parameter.

\begin{theorem}
\label{estimator1}
Let centered $(X_t)_{t\in\mathbb{Z}}$ be of the form \eqref{form} and
let $N\in\mathbb{N}$ be fixed. Assume that $r(m)=0$ for every $m\geq
N$. If $\gamma(N-1)\neq0$, then for every $n\geq N$, we have
\begin{equation*}
\phi= \frac{\gamma(n)}{\gamma(n-1)}.
\end{equation*}
In particular, $\gamma$ admits an exponential decay for $n\geq N$.
\end{theorem}
\begin{proof}
Let $\gamma(N-1)\neq0$. It follows directly from \eqref{sonta} and the
assumptions that
\begin{equation*}
\gamma(n) = \phi\gamma(n-1)\quad\text{for every}\ n\geq N.
\end{equation*}
The condition $\gamma(N-1)\neq0$ now implies the claim.
\end{proof}
Recall that the representation \eqref{langevin} is unique only after
$H$ is fixed. As a simple corollary for Theorem \ref{estimator1} we
obtain the following result giving some new information about the
uniqueness of the representation \eqref{langevin}.
\begin{kor}
Let $X$ be a strictly stationary process with a non-vanishing
autocovariance. Then there exists at most one pair $(H, G)$ satisfying
\eqref{langevin} such that the non-zero part of the autocovariance
function of the increment process $(\Delta_t G)_{t\in\mathbb{Z}}$ is finite.
\end{kor}
\begin{proof}
Assume that there exists $H_1, H_2 > 0$, and $G_1\in\mathcal{G}_{H_1}$
and $G_2\in\mathcal{G}_{H_2}$ such that the pairs $(H_1, G_1)$ and
$(H_2, G_2)$ satisfy \eqref{langevin} and the autocovariances of
$(\Delta_t G_1)_{t\in\mathbb{Z}}$ and $(\Delta_t G_2)_{t\in\mathbb{Z}}$
have cut-off points. From Theorem \ref{estimator1} it follows that $H_1
= H_2$ and since for a fixed $H$ the process $G$ in \eqref{langevin} is
unique, we get $G_1 = G_2$.
\end{proof}

\section{Estimation}\label{estimation}
Corollary \ref{estimator2} gives natural estimators for $\phi$ provided
that we have been able to choose between \eqref{solution} and \eqref
{solution2}, and that a value of $r(n)$ is known. We emphasize that in
our model it is sufficient to know only one (or in some cases two) of
the values $r(n)$, whereas in conventional ARMA modeling much stronger
assumptions are required. (In fact, in conventional ARMA modeling the
noise process is assumed to be white noise.) It is also worth to
mention that, generally, estimators of the parameters of stationary
processes are not expressible in a closed form. For example, this is
the case with the maximum likelihood and least squares estimators of
conventionally modeled ARMA processes, see \citep{brockwell}. Within
our method, the model fitting is simpler. 
Finally, it is worth to note that
assumption of one known value of $r(n)$ is a natural one
and cannot be avoided. Indeed, this is a direct consequence of the fact
that the pair $(\phi,Z)$ in representation \eqref{form} is not unique.
In fact, for practitioner, it is not absolutely necessary to know any
values of $r(n)$. The practitioner may make an educated guess and
proceed in estimation. If the obtained estimate then turns out to be
feasible, the practitioner can stop there. If the obtained estimate
turns out to be unreasonable (not on the interval $(0,1)$), then the
practitioner have to make another educated guess. The process is
similar to selecting $p$ and $q$ in traditional ARMA$(p,q)$ modeling.

Throughout this section, we assume that $(X_1, \ldots, X_T)$ is an
observed series from a centered strictly stationary process that is
modeled using the representation \eqref{form}. We use $\ha(n)$ to
denote an estimator of the corresponding autocovariance $\gamma(n)$.
For example, $\ha(n)$ can be given by
\begin{equation*}
\ha(n) = \frac{1}{T} \sum_{t=1}^{T-n}
X_tX_{t+n},
\end{equation*}
or more generally
\begin{equation*}
\ha(n) = \frac{1}{T} \sum_{t=1}^{T-n}
(X_t-\bar{X} ) (X_{t+n}-\bar{X} ),
\end{equation*}
where $\bar{X}$ is the sample mean of the observations.
For this estimator the corresponding sample covariance (function)
matrix is positive semidefinite. On the other hand, the estimator is
biased while it is asymptotically unbiased. Another option is to use
$T-n-1$ as a denominator. In this case one has an unbiased estimator,
but the sample covariance (function) matrix is no longer positive
definite. Obviously, both estimators have the same asymptotic
properties. Furthermore, for our purposes it is irrelevant how the
estimators $\ha(n)$ are defined, as long as they are consistent, and
the asymptotic distribution is known.

We next consider estimators of the parameter $\phi$ arising from
characterization \eqref{form}. In this context, we pose some
assumptions related to the autocovariance function of the observed
process $X$. The justification and testing of these assumptions are
discussed in Section~\ref{subsec:testing}. From \emph{a priori}
knowledge that $\phi\in(0,1)$ we enforce also the estimators to the
corresponding closed interval. However, if one prefers to use unbounded
versions of the estimators, one may very well do that. The asymptotic
properties are the same in both cases. We begin by defining an
estimator corresponding to the second part (2) of Corollary \ref{estimator2}.

\begin{defi}
\label{estimaattori1}
Assume that $\gamma(N) = 0$. Then we define
\begin{equation}
\label{eq:estimator_simplecase} \hat{\phi}_T = -\frac{r(N)}{\hat{\gamma}_T(N+1) + \hat{\gamma
}_T(N-1)}
\mathbh{1}_{\ha(N+1)+\ha(N-1) \neq0}
\end{equation}
whenever the right-hand side lies on the interval $[0,1]$. If the
right-hand side is below zero, we set $\hat{\phi}_T = 0$ and if the
right-hand side is above one, we set $\hat{\phi}_T = 1$.
\end{defi}

\begin{theorem}
\label{cons1}
Assume that $\gamma(N) =0$ and $r(N)\neq0$. If the
vector-valued estimator\\ $ [\ha(N+1), \ha(N-1) ]^\top$ is
consistent, then $\hat{\phi}_T$ is consistent.
\end{theorem}
\begin{proof}
Since $\gamma(N) = 0$ and $r(N) \neq0$, Equation \eqref{quadratic}
guarantees that $\gamma(N+1) + \gamma(N-1) \neq0$. Therefore
consistency of $\hat{\phi}_T$ follows directly from the continuous
mapping theorem.
\end{proof}

\begin{theorem}
\label{asym3}
Let $\hat{\phi}_T$ be given by (\ref{estimaattori1}), and assume that
$\gamma(N) = 0$ and $r(N) \neq0$. Set $\pmb{\gamma} =  [\gamma
(N+1), \gamma(N-1) ]^\top$ and $\hat{\pmb{\gamma}}_T =  [\ha
(N+1),\ha(N-1) ]^{\top}$. If
\begin{equation*}
l(T) (\hat{\pmb{\gamma}}_T - \pmb{\gamma} )\overset{\text {law}} {
\longrightarrow} \mathcal{N} (\pmb{0}, \varSigma )
\end{equation*}
for some covariance matrix $\varSigma$ and some rate function $l(T)$, then
\begin{equation*}
l(T) (\hat{\phi}_T - \phi ) \overset{\text {law}} {\longrightarrow}
\mathcal{N} \bigl(\pmb{0}, \nabla f(\pmb{\gamma })^\top\varSigma\nabla
f(\pmb{\gamma}) \bigr),
\end{equation*}
%
where $\nabla f(\pmb{\gamma})$ is given by
\begin{equation}
\label{deltaf} \nabla f(\pmb{\gamma}) = -\frac{r(N)}{ (\gamma(N+1) + \gamma
(N-1) )^2}\cdot %
\begin{bmatrix}
1\\
1
\end{bmatrix} %
.
\end{equation}
\end{theorem}
\begin{proof}
For the simplicity of notation, in the proof we use the unbounded
version of the estimator $\hat{\phi}_T$. Since the true value of $\phi$
lies strictly between 0 and 1, the very same result holds also for the
bounded estimator of Definition \ref{estimator1}. Indeed, this is a
simple consequence of the Slutsky's theorem. To begin with, let us
define an auxiliary function $f$ by
\begin{equation*}
f(\pmb{x}) = f(x_1,x_2) = \frac{r(N)}{x_1 + x_2}
\mathbh{1}_{x_1+x_2
\neq0}.
\end{equation*}
If $x_1+x_2\neq0$, the function $f$ is smooth in a neighborhood of
$\pmb{x}$. Since $\gamma(N) = 0$ together with $r(N) \neq0$ implies
that $\gamma(N+1) + \gamma(N-1) \neq0$, we may apply the delta method
at $\pmb{x} = \pmb{\gamma}$ to obtain
%
\begin{equation*}
l(T) (\hat{\phi}_T - \phi ) = -l(T) \bigl(f(\hat{\pmb{\gamma
}}_T) - f(\pmb{\gamma}) \bigr)\overset{\text{law}} {\longrightarrow}
\mathcal{N} \bigl(\pmb{0}, \nabla f(\pmb{\gamma})^\top\varSigma\nabla
f(\pmb {\gamma}) \bigr),
\end{equation*}
where $\nabla f(\pmb{\gamma})$ is given by \eqref{deltaf}.
This concludes the proof.
\end{proof}
\begin{rem}
By writing
\begin{equation*}
\varSigma= %
\begin{bmatrix}
\sigma_X^2 &\sigma_{XY}\\
\sigma_{XY} & \sigma_Y^2
\end{bmatrix} %
\end{equation*}
the variance of the limiting random variable reads
\begin{equation*}
\frac{r(N)^2}{ (\gamma(N+1) + \gamma(N-1) )^4}\bigl(\sigma_X^2 + 2
\sigma_{XY} + \sigma_Y^2\bigr).
\end{equation*}
\end{rem}
\begin{rem}
In many cases the convergency rate is the best possible, that is $l(T)
= \sqrt{T}$. However, our results are valid with any rate function. One
might, for example in the case of many long memory processes, have
other convergency rates for the estimators $\ha(n)$.
\end{rem}

We continue by defining an estimator corresponding to the first part
(1) of the Corollary \ref{estimator2}. For this we assume that, for
reasons discussed in Section~\ref{model}, we have chosen the solution
\eqref{solution} (cf. Remark \ref{rem:minus} and Section~\ref{subsec:testing}).
As above, we show that consistency and asymptotic
normality follow from the same properties of the autocovariance
estimators. In the sequel we use a short notation
\begin{equation}
\label{g} g(\pmb{x}) = g(x_1,x_2,x_3)
= (x_1+x_3 )^2-4x_2
\bigl(x_2-r(N)\bigr).
\end{equation}
In addition, we denote
\[
\pmb{\gamma} = \bigl[\gamma(N+1), \gamma(N), \gamma(N-1) \bigr]^\top
\]
and
\[
\hat{\pmb{\gamma}}_T = \bigl[\ha(N+1),\ha(N),\ha(N-1)
\bigr]^\top.
\]

\begin{defi}
\label{estimaattori2}
Assume that $\gamma(N) \neq0$. We define an estimator for $\phi$
associated to \eqref{solution} by
\begin{equation}
\label{est2} \hat{\phi}_T = \frac{\hat{\gamma}_T(N+1) + \hat{\gamma}_T(N-1) + \sqrt
{g (\hat{\pmb{\gamma}}_T )}\mathbh{1}_{g (\bo_T
)>0}}{2\hat{\gamma}_T(N)}
\mathbh{1}_{\ha(N) \neq0}
\end{equation}
whenever the right-hand side lies on the interval $[0,1]$. If the
right-hand side is below zero, we set $\hat{\phi}_T = 0$ and if the
right-hand side is above one, we set $\hat{\phi}_T = 1$.
\end{defi}

\begin{theorem}
\label{cons2}
Assume that $\gamma(N)\neq0$ and $g(\pmb{\gamma})>0$. Furthermore,
assume that $\phi$ is given by \eqref{solution}. If $\hat{\pmb{\gamma
}}_T$ is consistent, then $\hat{\phi}_T$ is consistent.
\end{theorem}
\begin{proof}
As $g(\pmb{\gamma})>0$, the result is again a simple consequence of the
continuous mapping theorem.
\end{proof}
Before proving the asymptotic normality, we present some short
notation. We set
\begin{equation}
\label{C} C_N = \frac{\gamma(N+1)+\gamma(N-1)+\sqrt{g(\pmb{\gamma})}}{\gamma(N)}
\end{equation}
and
\begin{equation}
\label{variance} %
\begin{split}\varSigma_\phi&=
\frac{1}{4\gamma(N)^2} \left( \bigl(\nabla\sqrt{g(\pmb {\gamma})}
\bigr)^\top\varSigma\nabla\sqrt{g(\pmb{\gamma})}+ 2 %
\begin{bmatrix}
1\\
-C_N\\
1
\end{bmatrix} %
^\top\varSigma \nabla
\sqrt{g(\pmb{\gamma})}
\right.\\
& \quad\left.+ %
\begin{bmatrix}
1\\
-C_N\\
1
\end{bmatrix} %
^\top
\varSigma %
\begin{bmatrix}
1\\
-C_N\\
1
\end{bmatrix} %
 \right), \end{split}
\end{equation}
where
\begin{equation*}
\nabla\sqrt{g(\pmb{\gamma})} = \frac{1}{\sqrt{g(\pmb{\gamma})}} %
\begin{bmatrix}
\gamma(N+1) + \gamma(N-1)\\
2(r(N) - 2\gamma(N))\\
\gamma(N+1) + \gamma(N-1)
\end{bmatrix} %
.
\end{equation*}

\begin{theorem}
\label{asym2}
Let the assumptions of Theorem \ref{cons2} prevail. If
\begin{equation*}
l(T) (\hat{\pmb{\gamma}}_T - \pmb{\gamma} )\overset{\text {law}} {
\longrightarrow} \mathcal{N} (\pmb{0}, \varSigma )
\end{equation*}
for some covariance matrix $\varSigma$ and some rate function $l(T)$, then
$l(T) (\hat{\phi}_T - \phi )$ is asymptotically normal with
zero mean and variance given by
\eqref{variance}.
\end{theorem}
\begin{proof}
The proof follows the same lines as the proof of Theorem \ref{asym3}
but for the reader's convenience, we present the details. Furthermore,
as in the proof of Theorem \ref{asym3}, since the true value of $\phi$
lies strictly between 0 and 1,
for the notational simplicity,
we may and will use the unbounded version of the estimator. Indeed, the
asymptotics for the bounded version then follow directly from the
Slutsky's theorem. We have\vadjust{\goodbreak}
\begin{align*}
& \biggl(\frac{\hat{\gamma}_T(N+1)\mathbh{1}_{\ha(N) \neq0}}{\ha(N)}- \frac{\gamma(N+1)}{\gamma(N)} \biggr)
\\
&\quad=\frac{1}{\ha(N)} \bigl(\ha(N+1)\mathbh{1}_{\ha(N) \neq0} - \gamma (N+1)
\bigr) + \biggl(\frac{\gamma(N+1)}{\ha(N)} - \frac{\gamma
(N+1)}{\gamma(N)} \biggr)
\\
&\quad= \frac{1}{\ha(N)} \biggl( \ha(N+1)\mathbh{1}_{\ha(N) \neq0} - \gamma (N+1) -
\frac{\gamma(N+1)}{\gamma(N)} \bigl(\ha(N)- \gamma(N) \bigr) \biggr).
\end{align*}
Similarly
\begin{align*}
& \biggl(\frac{\hat{\gamma}_T(N-1)\mathbh{1}_{\ha(N) \neq0}}{\ha(N)}- \frac{\gamma(N-1)}{\gamma(N)} \biggr)
\\
&\quad= \frac{1}{\ha(N)} \biggl( \ha(N-1)\mathbh{1}_{\ha(N) \neq0} - \gamma (N-1) -
\frac{\gamma(N-1)}{\gamma(N)} \bigl(\ha(N)- \gamma(N) \bigr) \biggr)
\end{align*}
and
\begin{align*}
& \biggl(\frac{\sqrt{g(\hat{\pmb{\gamma}}_T)}\mathbh{1}_{g (\bo
_T )>0}\mathbh{1}_{\ha(N) \neq0}}{\ha(N)}- \frac{\sqrt{g(\pmb
{\gamma})}}{\gamma(N)} \biggr)
\\
&\quad= \frac{1}{\ha(N)} \biggl( \sqrt{g(\hat{\pmb{\gamma}}_T)}\mathbh
{1}_{g (\bo_T )>0}\mathbh{1}_{\ha(N) \neq0} - \sqrt{g(\pmb {\gamma})} -
\frac{\sqrt{g(\pmb{\gamma})}}{\gamma(N)} \bigl(\ha(N)- \gamma(N) \bigr) \biggr).
\end{align*}
For $C_N$ given in \eqref{C} we have
\begin{equation*}
\begin{split} l(T) (\hat{\phi}_T-\phi ) =&\
\frac{l(T)}{2\hat{\gamma
}(N)} \bigl(\ha(N+1)\iii- \gamma(N+1)
\\
&+ \ha(N-1)\iii- \gamma(N-1) -C_N \bigl(\ha(N)-\gamma(N) \bigr)
\\
&+ \sqrt{g(\hat{\pmb{\gamma }}_T)}\ii\iii- \sqrt{g(\pmb{\gamma})}
\bigr). \end{split} %
\end{equation*}
By defining
\begin{equation*}
h(\pmb{x}) = h(x_1,x_2,x_3) =
\bigl(x_1 + x_3 + \sqrt{g(\pmb{x})}\mathbh
{1}_{g(\pmb{x}) >0} \bigr)\mathbh{1}_{x_2\neq0} - C_N
x_2
\end{equation*}
we have
\begin{equation}
\label{slut} l(T) (\hat{\phi}_T-\phi ) = \frac{l(T)}{2\hat{\gamma
}_T(N)}
\bigl(h (\hat{\pmb{\gamma}}_T ) - h(\pmb{\gamma}) \bigr).
\end{equation}
If $x_2 \neq0$ and $g(\pmb{x}) > 0$, the function $h$ is smooth in a
neighborhood of $\pmb{x}$. Therefore we may apply the delta method at
$\pmb{x} = \pmb{\gamma}$ to obtain
\begin{equation*}
l(T) \bigl(h (\hat{\pmb{\gamma}}_T ) - h(\pmb{\gamma}) \bigr)
\overset{\text{law}} {\longrightarrow} \mathcal{N} \bigl(\pmb{0}, \nabla h(\pmb{
\gamma})^\top\varSigma\nabla h(\pmb{\gamma}) \bigr),
\end{equation*}
where
\begin{align*}
\nabla h(\pmb{\gamma})^\top\varSigma\nabla h(
\pmb{\gamma}) &= \left( %
\begin{bmatrix}
1\\
-C_N\\
1
\end{bmatrix} %
+ \nabla
\sqrt{g(\pmb{\gamma})} \right)^\top\varSigma \left( %
\begin{bmatrix}
1\\
-C_N\\
1
\end{bmatrix} %
+ \nabla\sqrt{g(\pmb{\gamma})} \right)
\\
&= \bigl(\nabla\sqrt{g(\pmb{\gamma})} \bigr)^\top\varSigma\nabla\sqrt
{g(\pmb{\gamma})} + 2 %
\begin{bmatrix}
1\\
-C_N\\
1
\end{bmatrix} %
^\top
\varSigma \nabla\sqrt{g(\pmb{\gamma})}
\\
&\quad + %
\begin{bmatrix}
1\\
-C_N\\
1
\end{bmatrix} %
^\top \varSigma
\begin{bmatrix}
1\\
-C_N\\
1
\end{bmatrix} %
. 
\end{align*}
Hence \eqref{slut} and Slutsky's theorem imply that $l(T) (\hat{\phi
}_T - \phi )$ is asymptotically normal with zero mean and variance
given by \eqref{variance}.
\end{proof}
\begin{rem}
\label{rem:minus}
One straightforwardly observes the same limiting behavior as in
Theorems \ref{cons2} and \ref{asym2} for the estimator related to \eqref
{solution2}. This fact also can be used to determine which one of
Equations \eqref{solution} and \eqref{solution2} gives the correct $\phi
$ (cf. Section~\ref{subsec:testing}).
\end{rem}
\begin{rem}
\label{rem1}
If $\gamma(N) \neq0$ and $g(\pmb{\gamma}) = 0$ we may define an estimator
\begin{equation*}
\hat{\phi}_T = \frac{\ha(N+1) + \ha(N-1)}{2\ha(N)}\iii.
\end{equation*}
Assuming that
\begin{equation*}
l(T) (\hat{\pmb{\gamma}}_T - \pmb{\gamma} )\overset{\text {law}} {
\longrightarrow} \mathcal{N} (\pmb{0}, \varSigma )
\end{equation*}
it can be shown similarly as in the proofs of Theorems \ref{asym3} and
\ref{asym2} that
\begin{equation*}
 l(T) (\hat{\phi}_T - \phi ) \overset{\text {law}} {
\longrightarrow} \mathcal{N} \left(0, \frac{1}{4\gamma(N)^2} %
\begin{bmatrix}
1\\
-\frac{\gamma(N+1) + \gamma(N-1)}{\gamma(N)} \\
1
\end{bmatrix} %
^\top\varSigma %
\begin{bmatrix}
1\\
-\frac{\gamma(N+1) + \gamma(N-1)}{\gamma(N)} \\
1
\end{bmatrix} %
 \right)
\end{equation*}
\end{rem}

\begin{rem}
\label{rem2}
The estimator related to Theorem \ref{estimator1} reads
\begin{equation*}
\hat{\phi}_T = \frac{\hat{\gamma}_T(n+1)}{\hat{\gamma}_T(n)}\mathbh {1}_{\ha(n) \neq0},
\end{equation*}
where we assume that $\gamma(n)\neq0$.
By using the same techniques as earlier, it can be shown that if
\begin{equation*}
\small{ l(T) \bigl(\hat{\gamma}_T(n+1) - \gamma(n+1), \hat{
\gamma}_T(n)- \gamma (n) \bigr) \overset{\text{law}} {
\longrightarrow} \mathcal{N} \left( \pmb{0}, %
\begin{bmatrix}
\sigma_X^2 & \sigma_{XY} \\
\sigma_{XY} & \sigma_Y^2
\end{bmatrix}
 \right), }
\end{equation*}
then
\begin{equation*}
l(T) (\hat{\phi}_T - \phi ) \overset{\text {law}} {\longrightarrow}
\mathcal{N} \biggl(0, \frac{\sigma_X^2}{\gamma
(n)^2} + \frac{\gamma(n+1)^2}{\gamma(n)^4}
\sigma_Y^2 - 2\frac{\gamma
(n+1)}{\gamma(n)^3}\sigma_{XY}
\biggr).
\end{equation*}
\end{rem}

Note that the asymptotics given in Remarks \ref{rem1} and \ref{rem2}
hold also if one forces the corresponding estimators to the interval
$[0,1]$ as we did in Definitions \ref{estimaattori1} and \ref{estimaattori2}.

\subsection{Testing the underlying assumptions}
\label{subsec:testing}
When choosing the estimator that corresponds the situation at hand, we
have to make assumptions related to the values of $\gamma(N)$ (for some
$N$) and $g(\pmb{\gamma})$. In addition, we have to consider the
question of the choice between \eqref{solution} and \eqref{solution2}.\vadjust{\goodbreak}

Let us first discuss how to test the null hypothesis that $\gamma(N) =
0$. If the null hypothesis holds, then by asymptotic normality of the
autocovariances, we have that
\begin{equation}
\label{law1} l(T)\ha(N)\overset{\text{law}} {\longrightarrow} \mathcal{N}
\bigl(0,\sigma^2 \bigr)
\end{equation}
with some $\sigma^2$. Hence we may use
\begin{equation*}
\ha(N) \sim_a \mathcal{N} \biggl(0, \frac{\sigma^2}{l(T)^2} \biggr)
\end{equation*}
as a test statistics. A similar approach can be applied also when
testing the null hypothesis that $g(\pmb{\gamma}) = 0$, where $g$ is
defined by \eqref{g}. The alternative hypothesis is of the form $g(\pmb
{\gamma}) > 0$. Assuming that the null hypothesis holds, we obtain by
the delta method that
\begin{equation*}
l(T) \bigl(g(\hat{\pmb{\gamma}}_T) - g(\pmb{\gamma}) \bigr)
\overset{\text {law}} {\longrightarrow} \mathcal{N} \bigl(0, \tilde{
\sigma}^2 \bigr)
\end{equation*}
for some $\tilde{\sigma}^2$ justifying the use of
\begin{equation*}
g(\hat{\pmb{\gamma}}_T) \sim_a \mathcal{N} \biggl(0,
\frac{\tilde{\sigma
}^2}{l(T)^2} \biggr)
\end{equation*}
as a test statistics. If the tests above suggest that $\gamma(N) \neq
0$ and $g(\pmb{\gamma})> 0$, then the choice of the sign can be based
on the discussion in Section~\ref{model}. Namely, if for the ratio $a_N = \frac
{r(N)}{\gamma(N)}$ it holds that $a_N \leq0$ or $a_N \geq1$, then the
sign is unambiguous. The sign of $\gamma(N)$ can be deduced from the
previous testing of the null hypothesis $\gamma(N) = 0$. By \eqref
{law1}, if necessary, one can test the null hypothesis $\gamma(N) =
r(N)$ using the test statistics
%
%
\begin{equation*}
\ha(N) \sim_a \mathcal{N} \biggl(r(N), \frac{\sigma^2}{l(T)^2} \biggr),
\end{equation*}
where the alternative hypothesis is of the form $\frac{r(N)}{\gamma(N)}
< 1$. Finally, assume that one wants to test if the null hypothesis
$a_N = a_k$ holds. By the delta method we obtain that
\begin{equation*}
l(T) (\hat{a}_N-\hat{a}_k - a_N
+a_k )\overset{\text {law}} {\longrightarrow} \mathcal{N} \bigl(0,
\bar{\sigma}^2 \bigr)
\end{equation*}
for some $\bar{\sigma}^2$ suggesting that
\begin{equation*}
\hat{a}_N-\hat{a}_k \sim_a \mathcal{N}
\biggl(0, \frac{\bar{\sigma
}^2}{l(T)^2} \biggr)
\end{equation*}
could be utilized as a test statistics.

\section{Simulations}
\label{simulations}
We present a simulation study to assess the finite sample performance
of the estimators. In the simulations, we apply the estimator
corresponding to the first part (1) of Corollary \ref{estimator2}. 
We simulate data from AR$(1)$ processes and ARMA$(1,2)$ processes with
$\theta_1 = 0.8$ and $\theta_2 = 0.3$ as the MA parameters. (Note that
these processes correspond to Examples \ref{exar} and \ref{exarma}.) We
assess the effects of the sample size $T$, AR$(1)$ parameter $\varphi$,
and the chosen lag $N$. We consider the sample sizes $T=50, 500, 5000,
50000$, lags $N=1, 2, 3, \ldots, 10$, and the true parameter values
$\varphi= 0.1, 0.2, 0.3, \ldots, 0.9$. For each combination, we simulate
1000 draws. The sample means of the obtained estimates are tabulated in
Appendix~\ref{simutables}.

Histograms given in Figures \ref{samplesize}, \ref{parameter} and \ref
{lag} reflect the effects of the sample size $T$, AR$(1)$ parameter
$\varphi$, and the chosen lag $N$, respectively. In Figure \ref
{samplesize}, the parameter $\varphi= 0.5$ and the lag $N = 3$. In
Figure \ref{parameter}, the sample size $T = 5000$ and the lag $N = 3$.
In Figure \ref{lag}, the parameter $\varphi= 0.5$ and the sample size
$T = 5000$. The summary statistics corresponding to the data displayed
in the histograms are given in Appendix~\ref{simutables}.

Figure \ref{samplesize} exemplifies the rate of convergence of the
estimator as the number of observations grows. One can see that with
the smallest sample size, the lower bound is hit numerous times due to
the large variance of the estimator. In the upper series of the
histograms, the standard deviation reduces from $0.326$ to $0.019$,
whereas in the lower series it reduces from $0.250$ to $0.008$. The
faster convergence in the case of ARMA$(1,2)$ can be explained with the
larger value of $\gamma(3)$ reducing the variance in comparison to the
AR$(1)$ case. The same phenomenon recurs also in the other two figures.

Figure \ref{parameter} reflects the effect of the AR$(1)$ parameter on
the value of $\gamma(3)$ and consequently on the variance of the
estimator. The standard deviation reduces from $0.322$ to $0.020$ in
the case of AR$(1)$ and from $0.067$ to $0.009$ in the case of ARMA$(1,2)$.

In Figure \ref{lag} one can see how an increase in the lag increases
the variance of the estimator. In the topmost sequence, the standard
deviation increases from $0.014$ to $0.326$ and in the bottom sequence
from $0.015$ to $0.282$.

We wish to emphasize that in general smaller lag does not imply smaller
variance, since the autocovariance function of the observed process is
not necessarily decreasing. In addition, although the autocovariance
$\gamma(N)$ appears to be the dominant factor when it comes to the
speed of convergence, there are also other possibly significant terms
involved in the limit distribution of Theorem \ref{asym2}.


\begin{figure}[h!]
\includegraphics{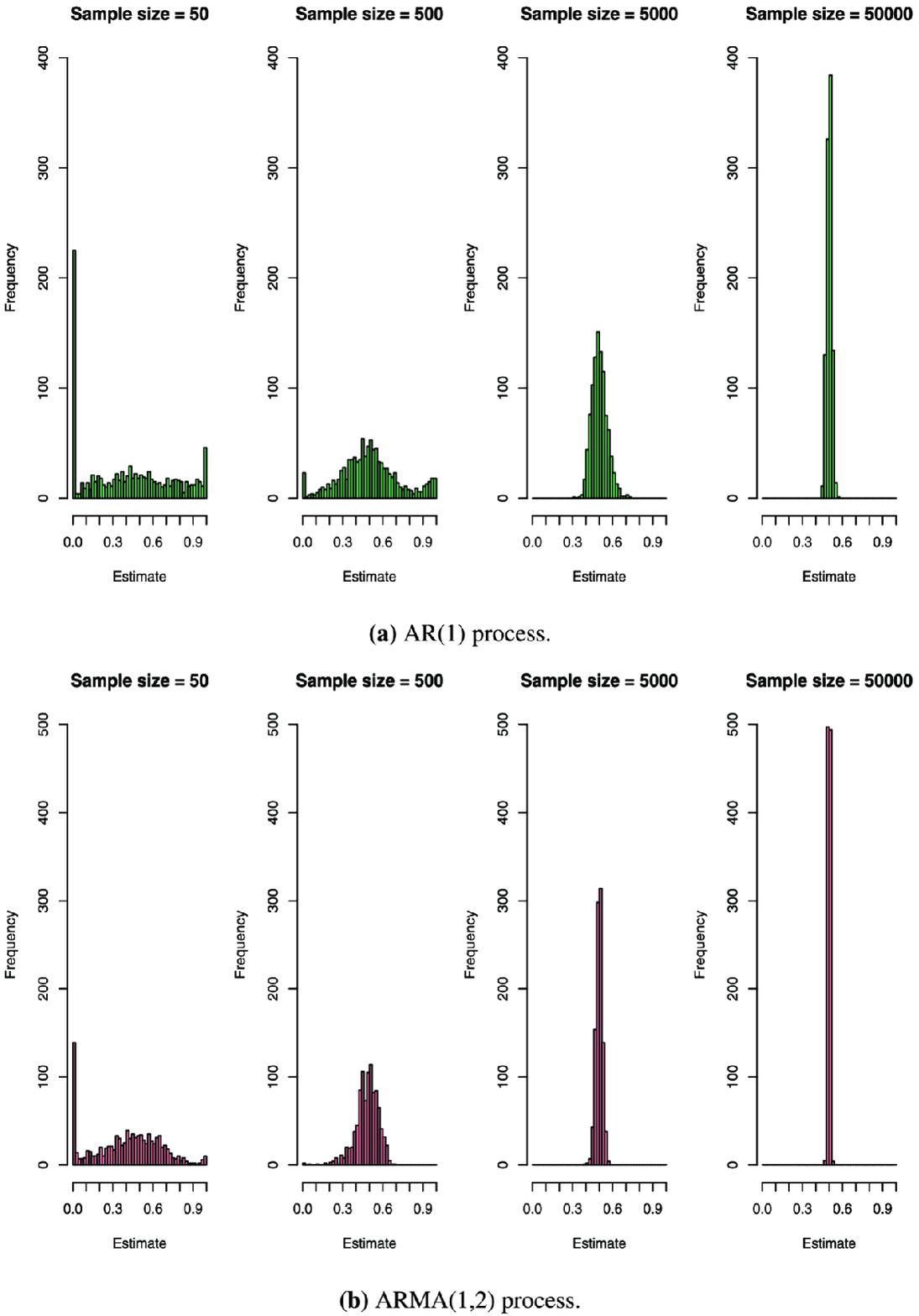}
\caption{The effect of the sample size $T$ on the estimates
$\hat{\varphi} = \hat{\phi}$. The true parameter value
$\varphi= 0.5$ and the lag $N = 3$. The number of iterations is 1000}
\label{samplesize}
\label{ARsample}
\label{ARMAsample}
\end{figure}

\begin{figure}[h!]
\includegraphics{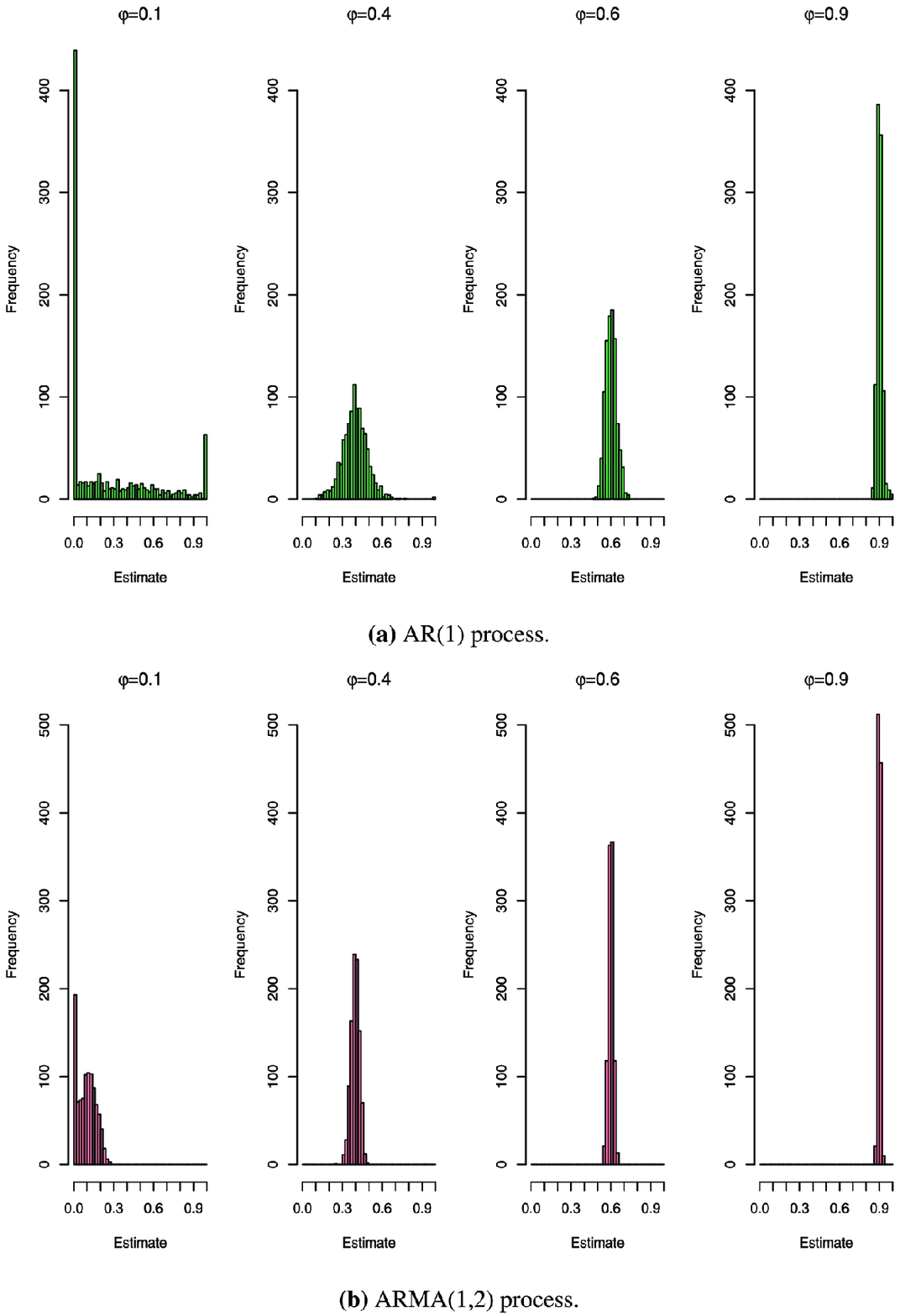}
\caption{The effect of the true parameter value
$\varphi$ on the estimates $\hat{\varphi} = \hat{\phi}$.
The sample size $T=5000$ and the lag $N = 3$. The number of iterations
is 1000}
\label{parameter}
\label{ARparameter}
\label{ARMAparameter}
\end{figure}

\begin{figure}[h!]
\includegraphics{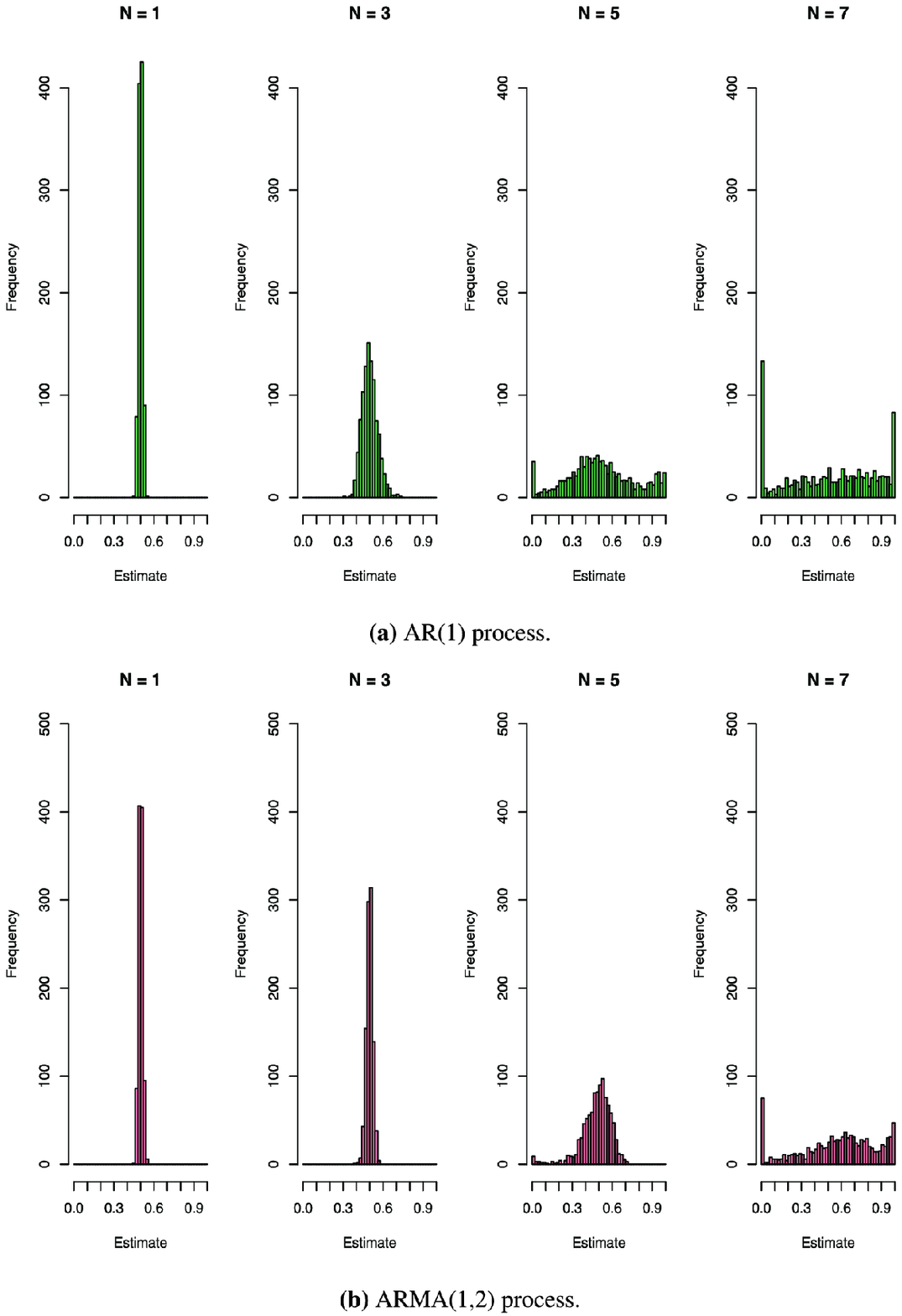}
\caption{The effect of the lag $N$ on the estimates
$\hat{\varphi} = \hat{\phi}$. The sample size $T=5000$ and the
true parameter value $\varphi= 0.5$. The number of iterations is 1000}
\label{lag}
\label{ARlag}
\label{ARMAlag}
\end{figure}

\begin{appendix}
\section{Proof of Theorem \ref{main}}\label{techproof}
We provide here a detailed proof of Theorem \ref{main}. The continuous
time version of the theorem was recently proved in \citep{asymptotics}
and we loosely follow the same lines in our proof for the discrete time version.
\begin{defi}
Let $H>0$. A discrete time stochastic process $Y = (Y_{e^t})_{t\in
\mathbb{Z}}$ with $\lim_{t\to-\infty} Y_{e^t} = 0$ is H-self-similar if
\begin{equation*}
(Y_{e^{t+s}} )_{t\in\mathbb{Z}} \overset{\text{law}} {=}
\bigl(e^{sH} Y_{e^t} \bigr)_{t\in\mathbb{Z}}
\end{equation*}
for every $s\in\mathbb{Z}$ in the sense of finite-dimensional distributions.
\end{defi}

\begin{defi}
Let $H>0$. In addition, let $X = (X_t)_{t\in\mathbb{Z}}$ and $Y =
(Y_{e^t})_{t\in\mathbb{Z}}$ be stochastic processes. We define the
discrete Lamperti transform by
\begin{equation*}
(\mathcal{L}_H X)_{e^t} = e^{tH}X_{t}
\end{equation*}
and its inverse by
\begin{equation*}
\bigl(\mathcal{L}_{H}^{-1} Y\bigr)_t =
e^{-tH}Y_{e^t}.
\end{equation*}
\end{defi}

\begin{theorem}[Lamperti \cite{lamperti1962semi}]
\label{lamperti}
If $X = (X_t)_{t\in\mathbb{Z}}$ is strictly stationary, then $(\mathcal
{L}_H X)_{e^t}$ is H-self-similar. Conversely, if $Y= (Y_{e^t})_{t\in
\mathbb{Z}}$ is H-self-similar, then $(\mathcal{L}_H^{-1} Y)_t$ is
strictly stationary.
\end{theorem}

\begin{lemma}
\label{Y}
Let $H>0$ and assume that $(Y_{e^t})_{t\in\mathbb{Z}}$ is
H-self-similar. Let us denote $\Delta_t Y_{e^t} = Y_{e^t} -
Y_{e^{t-1}}$. Then the process $(G_t)_{t\in\mathbb{Z}}$ defined by
\begin{equation}
G_t = \lleft\{ %
\begin{array}{@{}ll}
\sum_{k=1}^t e^{-kH} \Delta_k Y_{e^k}, \quad&t\geq1 \\
0, \quad&t=0 \\
-\sum_{k=t+1}^0 e^{-kH} \Delta_k Y_{e^k}, \quad&t\leq-1
\end{array} %
\rright.
\end{equation}
belongs to $\mathcal{G}_H$.
\begin{proof}
By studying the cases $t\geq2, t=1, t=0$ and $t\leq-1$ separately, it
is straightforward to see that
\begin{equation}
\label{difference} \Delta_t G = e^{-tH} \Delta_t
Y_{e^t}\quad\text{for every}\ t\in\mathbb{Z}.
\end{equation}
Now
\[
\lim_{k\to-\infty} \sum_{t = k}^0
e^{tH} \Delta_t G = \lim_{k\to
-\infty} \sum
_{t=k}^0 \Delta_t
Y_{e^t} = Y_{e^0} - \lim_{k\to-\infty}
Y_{e^k}
\]
and since $Y$ is self-similar, we have
\[
Y_{e^k} \law e^{kH} Y_{e^0}.
\]
Thus
\[
\lim_{k\to-\infty} Y_{e^k} = 0
\]
in distribution, and hence also in probability. This implies that
\[
\sum_{t = -\infty}^0 e^{tH}
\Delta_t G
\]
is an almost surely finite random variable. Next we show that $G$ has
strictly stationary increments. For this, assume that $t,s,l \in\mathbb
{Z}$ with $t>s$ are arbitrary. Then
\begin{equation*}
\begin{split} G_t - G_s &= \sum
_{k=s+1}^t \Delta_k G= \sum
_{k=s+1}^t e^{-kH}\Delta_k
Y_{e^k} = \sum_{j = s+l+1}^{t+l}
e^{-(j-l)H}\Delta_{j-l} Y_{e^{j-l}}
\\
&\overset{\text{law}} {=} \sum_{j = s+l+1}^{t+l}
e^{-jH} \Delta_j Y_{e^j} = G_{t+l} -
G_{s+l}, \end{split} %
\end{equation*}
where the equality in law follows from $H$-self-similarity of
$(Y_{e^t})_{t\in\mathbb{Z}}$. Treating $n$-dimensional vectors
similarly concludes the proof.
\end{proof}
\end{lemma}

\begin{proof}[Proof of Theorem \ref{main}]
Assume first that $X$ is strictly stationary. In this case $X$ clearly
satisfies the limit condition. In addition, there exists a
H-self-similar $Y$ such that
\begin{equation*}
\begin{split} \Delta_{t} X &= e^{-tH}
Y_{e^t} - e^{-(t-1)H}Y_{e^{t-1}}
\\
&= \bigl(e^{-H} - 1 \bigr)e^{-(t-1)H} Y_{e^{t-1}} +
e^{-tH}(Y_{e^t} - Y_{e^{t-1}})
\\
&= \bigl(e^{-H} - 1 \bigr)X_{t-1} + e^{-tH}
\Delta_t Y_{e^t}. \end{split} %
\end{equation*}
Defining the process $G$ as in Lemma \ref{Y} completes the proof of the
`if' part.
For the proof of the `only if' part, assume that $G\in\mathcal{G}_H$.
From \eqref{langevin} it follows that
\begin{equation*}
\begin{split} X_t &= e^{-H}X_{t-1} +
\Delta_tG = e^{-2H}X_{t-2} + e^{-H}
\Delta_{t-1}G + \Delta_t G
\\
&=\sum_{j=0}^n e^{-jH}
\Delta_{t-j} G + e^{-(n+1)H}X_{t-n-1}
\\
&= e^{-tH} \Biggl(\sum_{k = t-n}^t
e^{kH} \Delta_k G + e^{(t-n-1)H}X_{t-n-1}
\Biggr) \end{split} %
\end{equation*}
for every $n\in\mathbb{N}$. Since $G\in\mathcal{G}_H$ and $\lim_{m\to
-\infty} e^{mH} X_{m} = 0$ in probability, we obtain that
\begin{equation*}
\label{Xform} X_t = e^{-tH} \sum
_{k= -\infty}^t e^{kH} \Delta_k G
\end{equation*}
for every $t\in\mathbb{Z}$. Now, by strictly stationary increments of
$G$, we have
\begin{equation*}
e^{-tH}\sum_{j=-M}^t
e^{jH}\Delta_{j+s} G \overset{\text{law}} {=}
e^{-tH}\sum_{j=-M}^t
e^{jH} \Delta_j G.
\end{equation*}
for every $t, M\in\mathbb{Z}$ such that $-M \leq t$.
Since the sums above converge as $M$ tends to infinity, we obtain
%
\begin{equation*}
\begin{split} X_{t+s} = e^{-(t+s)H} \sum
_{ j =-\infty}^t e^{(j+s)H} \Delta_{j+s} G
\overset{\text{law}} {=} e^{-tH}\sum_{j=-\infty}^t
e^{jH} \Delta_j G = X_t. \end{split}
\end{equation*}
Treating multidimensional distributions similarly we thus observe that
$X$ is strictly stationary.
Finally, to prove the uniqueness assume there exist $G_1, G_2\in\mathcal
{G}_H$ such that
\begin{equation*}
\begin{split} e^{tH} X_t = \sum
_{k=-\infty}^t e^{kH} \Delta_k
G_1 = \sum_{k=-\infty
}^t
e^{kH} \Delta_k G_2 \end{split} %
\end{equation*}
for every $t\in\mathbb{Z}$. Then
\begin{equation*}
e^{tH} X_t - e^{(t-1)H}X_{t-1} =
e^{tH} \Delta_t G_1 = e^{tH}
\Delta_t G_2.
\end{equation*}
Hence $\Delta_t G_1 = \Delta_t G_2$ for every $t\in\mathbb{Z}$ implying
that $G_1 = G_2 + c$. Since both processes are zero at $t=0$, it must
hold that $c = 0$.
\end{proof}

\begin{rem}
Corollary \ref{cor:stat} is almost trivial. However, it is well
motivated by Theorem \ref{main}. On the other hand, Theorem \ref{main}
is far away from trivial as it states both sufficient\vadjust{\goodbreak} and necessary
conditions. We prove Theorem \ref{main} using discrete Lamperti
transform. In principle, one could consider proving Theorem \ref{main}
by starting from Corollary \ref{cor:stat}. However, at this point, we
have not assumed any moment conditions, and thus it is not clear
whether a process $G$ constructed from $Z^{(H)}$ of Corollary \ref
{cor:stat} would satisfy $G\in\mathcal{G}_H$. Indeed, a counter example
is provided in \cite[Proposition 2.1.]{asymptotics}. See also \cite
[Theorem 2.2.]{asymptotics}, where moment conditions are discussed.
\end{rem}

\section{Discussion on special cases}\label{worstcase}

In this appendix we take a closer look at ``worst case scenario''
processes related to the choice between \eqref{solution} and \eqref
{solution2}. These are such processes that, for some $0<a<1$, $a_j = a$
for every $j\in\mathbb{Z}$. By \eqref{quadratic} this is equivalent to
\begin{equation}
\label{b19} \frac{\gamma(j+1)+\gamma(j-1)}{\gamma(j)} = b
\end{equation}
for every $j\in\mathbb{Z}$, where $\phi<b<\phi+ \frac{1}{\phi}$. In
order to study processes of this form, we consider formal power series.

\begin{defi}
Let
\begin{equation*}
f(x) = \sum_{n=0}^\infty
c_nx^n
\end{equation*}
be a formal power series in $x$. We now define the coefficient
extractor operator $[\cdot]\{*\}$ by
\begin{equation*}
\bigl[x^m\bigr]\bigl\{f(x)\bigr\} = c_m
\end{equation*}
\end{defi}
Setting $j=0$ in \eqref{b19} we obtain that $\gamma(1) = \frac{b}{2}\gamma
(0)$. This leads to the following recursion.
\begin{equation}
\label{eq:recursion} \gamma(n) = b\gamma(n-1) - \gamma(n-2)\quad\text{for}\ n\geq2.
\end{equation}
It follows immediately from the first step of the recursion that $b>2$
does not define an autocovariance function of a stationary process.
Note also that for $b=2$ Equation \eqref{eq:recursion} implies that
$\gamma(n)=\gamma(0)$ for every $n\in\mathbb{Z}$. This corresponds to
the completely degenerate process $X_n = X_0$. We next study the case
$0 <b<2$.
For this, we define a generating function regarded as a formal power
series by
\begin{equation}
\label{generating} f(x) = \sum_{n=0}^\infty
\gamma(n) x^n.
\end{equation}
Then the coefficients of $f(x)$ satisfy
\begin{equation*}
\begin{split} \bigl[x^n\bigr]\bigl\{f(x)\bigr\} &= b
\bigl[x^{n-1}\bigr]\bigl\{f(x)\bigr\} - \bigl[x^{n-2}\bigr]\bigl
\{f(x)\bigr\}
\\
&=\bigl[x^n\bigr]\bigl\{bxf(x)\bigr\} - \bigl[x^n\bigr]
\bigl\{x^2f(x)\bigr\}
\\
&= \bigl[x^n\bigr]\bigl\{bxf(x)-x^2f(x)\bigr\}
\end{split} %
\end{equation*}
for $n\geq2$. For simplicity, we assume that $\gamma(0) = 1$. By taking
the constant and the first order terms into account we obtain
\begin{equation*}
f(x) = bxf(x)-x^2f(x)-bx+1+\frac{b}{2}x,\vadjust{\goodbreak}
\end{equation*}
which implies
\begin{equation*}
f(x) = \frac{1-\frac{b}{2}x}{x^2-bx+1}.
\end{equation*}
Since the function above is analytic at $x=0$, the corresponding power
series expansion is \eqref{generating}. Furthermore, since the
recursion formula is linear, for a general $\gamma(0)$ it holds that
\begin{equation*}
\gamma(n) = \gamma(0)\bigl[x^n\bigr] \Biggl\{ \biggl(1-
\frac{b}{2}x \biggr)\sum_{n=0}^\infty
\bigl(bx-x^2\bigr)^n \Biggr\}.
\end{equation*}



\section{Tables}\label{simutables}
The simulation results highlighted in Section~\ref{simulations} are
chosen from a more extensive set of simulations. All the simulation
results are given in a tabulated form in this appendix. The two
processes considered in the simulations are AR$(1)$ and ARMA$(1,2)$.
The used MA parameters are $\theta_1 = 0.8$ and $\theta_2 = 0.3$. The
tables represent the efficiency dependence of the estimator on the
AR$(1)$ parameter $\varphi$ and the used lag $N$. We have varied the
column variable AR$(1)$ parameter from $0.1$ to $0.9$ and the row
variable lag from $1$ to $10$. The tables display the sample means of
the estimates from $1000$ iterations with different sample sizes. At
the end of this appendix, we provide summary statistics tables
corresponding to the histograms presented in Section~\ref{simulations}.

\begin{table}[ht]
%
\begin{tabular}{rrrrrrrrrr}
\hline
$N/\varphi$& 0.1 & 0.2 & 0.3 & 0.4 & 0.5 & 0.6 & 0.7 & 0.8 & 0.9 \\
\hline
1 & 0.10 & 0.18 & 0.27 & 0.38 & 0.48 & 0.57 & 0.67 & 0.77 & 0.85 \\
2 & 0.25 & 0.26 & 0.30 & 0.35 & 0.45 & 0.54 & 0.64 & 0.74 & 0.82 \\
3 & 0.32 & 0.35 & 0.35 & 0.40 & 0.41 & 0.48 & 0.57 & 0.69 & 0.80 \\
4 & 0.30 & 0.37 & 0.42 & 0.47 & 0.50 & 0.52 & 0.55 & 0.66 & 0.77 \\
5 & 0.33 & 0.39 & 0.42 & 0.50 & 0.53 & 0.57 & 0.61 & 0.65 & 0.75 \\
6 & 0.34 & 0.37 & 0.42 & 0.47 & 0.56 & 0.60 & 0.66 & 0.69 & 0.74 \\
7 & 0.32 & 0.37 & 0.43 & 0.49 & 0.57 & 0.60 & 0.68 & 0.69 & 0.75 \\
8 & 0.32 & 0.34 & 0.45 & 0.51 & 0.57 & 0.64 & 0.69 & 0.72 & 0.76 \\
9 & 0.31 & 0.37 & 0.44 & 0.50 & 0.59 & 0.64 & 0.70 & 0.73 & 0.78 \\
10 & 0.34 & 0.35 & 0.43 & 0.51 & 0.58 & 0.64 & 0.70 & 0.75 & 0.78 \\
\hline
\end{tabular}
\caption{The sample means of the parameter estimates $\hat{\varphi} =
\hat{\phi}$ for AR$(1)$ processes with different parameter values
$\varphi$ using lags $N=1, 2, 3, \ldots, 10$. The sample size is 50 and
the number of iterations is 1000}\vspace*{-12pt}
\end{table}

\begin{table}[ht]
%
\begin{tabular}{rrrrrrrrrr}
\hline
$N/\varphi$& 0.1 & 0.2 & 0.3 & 0.4 & 0.5 & 0.6 & 0.7 & 0.8 & 0.9 \\
\hline
1 & 0.10 & 0.20 & 0.30 & 0.40 & 0.50 & 0.60 & 0.70 & 0.80 & 0.90 \\
2 & 0.23 & 0.24 & 0.30 & 0.40 & 0.51 & 0.60 & 0.70 & 0.80 & 0.91 \\
3 & 0.29 & 0.31 & 0.34 & 0.40 & 0.50 & 0.61 & 0.71 & 0.81 & 0.91 \\
4 & 0.32 & 0.37 & 0.40 & 0.40 & 0.49 & 0.61 & 0.70 & 0.81 & 0.90 \\
5 & 0.30 & 0.37 & 0.42 & 0.48 & 0.50 & 0.58 & 0.70 & 0.81 & 0.90 \\
6 & 0.30 & 0.36 & 0.44 & 0.47 & 0.53 & 0.58 & 0.68 & 0.80 & 0.90 \\
7 & 0.30 & 0.37 & 0.44 & 0.49 & 0.53 & 0.57 & 0.65 & 0.79 & 0.90 \\
8 & 0.32 & 0.39 & 0.44 & 0.51 & 0.57 & 0.61 & 0.68 & 0.76 & 0.90 \\
9 & 0.30 & 0.38 & 0.45 & 0.51 & 0.59 & 0.63 & 0.68 & 0.75 & 0.89 \\
10 & 0.32 & 0.39 & 0.46 & 0.52 & 0.58 & 0.64 & 0.70 & 0.75 & 0.89 \\
\hline
\end{tabular}
\caption{The sample means of the parameter estimates $\hat{\varphi} =
\hat{\phi}$ for AR$(1)$ processes with different parameter values
$\varphi$ using lags $N=1, 2, 3, \ldots, 10$. The sample size is 500 and
the number of iterations is 1000}\vspace*{-24pt}
\end{table}

\begin{table}[ht!]
%
\begin{tabular}{rrrrrrrrrr}
\hline
$N/\varphi$& 0.1 & 0.2 & 0.3 & 0.4 & 0.5 & 0.6 & 0.7 & 0.8 & 0.9 \\
\hline
1 & 0.10 & 0.20 & 0.30 & 0.40 & 0.50 & 0.60 & 0.70 & 0.80 & 0.90 \\
2 & 0.13 & 0.20 & 0.30 & 0.40 & 0.50 & 0.60 & 0.70 & 0.80 & 0.90 \\
3 & 0.26 & 0.27 & 0.32 & 0.40 & 0.50 & 0.60 & 0.70 & 0.80 & 0.90 \\
4 & 0.30 & 0.32 & 0.34 & 0.42 & 0.51 & 0.61 & 0.70 & 0.80 & 0.90 \\
5 & 0.29 & 0.37 & 0.38 & 0.43 & 0.51 & 0.62 & 0.71 & 0.80 & 0.90 \\
6 & 0.31 & 0.37 & 0.41 & 0.45 & 0.49 & 0.62 & 0.71 & 0.81 & 0.90 \\
7 & 0.29 & 0.38 & 0.40 & 0.47 & 0.52 & 0.59 & 0.72 & 0.81 & 0.90 \\
8 & 0.29 & 0.40 & 0.45 & 0.51 & 0.54 & 0.58 & 0.71 & 0.81 & 0.91 \\
9 & 0.32 & 0.37 & 0.41 & 0.50 & 0.54 & 0.60 & 0.68 & 0.82 & 0.91 \\
10 & 0.29 & 0.37 & 0.41 & 0.51 & 0.57 & 0.61 & 0.68 & 0.82 & 0.91 \\
\hline
\end{tabular}
\caption{The sample means of the parameter estimates $\hat{\varphi} =
\hat{\phi}$ for AR$(1)$ processes with different parameter values
$\varphi$ using lags $N=1, 2, 3, \ldots, 10$. The sample size is 5000 and
the number of iterations is 1000}
\end{table}

\begin{table}[ht!]
%
\begin{tabular}{rrrrrrrrrr}
\hline
$N/\varphi$ & 0.1 & 0.2 & 0.3 & 0.4 & 0.5 & 0.6 & 0.7 & 0.8 & 0.9 \\
\hline
1 & 0.10 & 0.20 & 0.30 & 0.40 & 0.50 & 0.60 & 0.70 & 0.80 & 0.90 \\
2 & 0.10 & 0.20 & 0.30 & 0.40 & 0.50 & 0.60 & 0.70 & 0.80 & 0.90 \\
3 & 0.21 & 0.21 & 0.30 & 0.40 & 0.50 & 0.60 & 0.70 & 0.80 & 0.90 \\
4 & 0.28 & 0.30 & 0.33 & 0.40 & 0.50 & 0.60 & 0.70 & 0.80 & 0.90 \\
5 & 0.29 & 0.34 & 0.36 & 0.41 & 0.51 & 0.60 & 0.70 & 0.80 & 0.90 \\
6 & 0.29 & 0.37 & 0.39 & 0.42 & 0.52 & 0.60 & 0.70 & 0.80 & 0.90 \\
7 & 0.29 & 0.37 & 0.44 & 0.45 & 0.51 & 0.61 & 0.70 & 0.80 & 0.90 \\
8 & 0.31 & 0.37 & 0.43 & 0.48 & 0.49 & 0.62 & 0.70 & 0.80 & 0.90 \\
9 & 0.31 & 0.35 & 0.43 & 0.49 & 0.53 & 0.60 & 0.71 & 0.80 & 0.90 \\
10 & 0.32 & 0.37 & 0.42 & 0.48 & 0.53 & 0.58 & 0.72 & 0.80 & 0.90 \\
\hline
\end{tabular}
\caption{The sample means of the parameter estimates $\hat{\varphi} =
\hat{\phi}$ for AR$(1)$ processes with different parameter values
$\varphi$ using lags $N=1, 2, 3, \ldots, 10$. The sample size is 50000 and
the number of iterations is 1000}
\end{table}

\begin{table}[ht!]
%
\begin{tabular}{rrrrrrrrrr}
\hline
$N/\varphi$& 0.1 & 0.2 & 0.3 & 0.4 & 0.5 & 0.6 & 0.7 & 0.8 & 0.9 \\
\hline
1 & 0.08 & 0.14 & 0.22 & 0.32 & 0.41 & 0.52 & 0.61 & 0.72 & 0.81 \\
2 & 0.09 & 0.13 & 0.20 & 0.30 & 0.39 & 0.50 & 0.60 & 0.72 & 0.82 \\
3 & 0.32 & 0.33 & 0.32 & 0.34 & 0.40 & 0.46 & 0.58 & 0.71 & 0.81 \\
4 & 0.65 & 0.66 & 0.62 & 0.60 & 0.60 & 0.56 & 0.60 & 0.68 & 0.78 \\
5 & 0.64 & 0.67 & 0.69 & 0.70 & 0.69 & 0.69 & 0.69 & 0.70 & 0.77 \\
6 & 0.64 & 0.68 & 0.69 & 0.72 & 0.74 & 0.75 & 0.76 & 0.75 & 0.78 \\
7 & 0.64 & 0.67 & 0.70 & 0.72 & 0.77 & 0.78 & 0.79 & 0.79 & 0.80 \\
8 & 0.65 & 0.67 & 0.71 & 0.72 & 0.76 & 0.79 & 0.81 & 0.80 & 0.83 \\
9 & 0.63 & 0.68 & 0.72 & 0.74 & 0.78 & 0.80 & 0.82 & 0.83 & 0.84 \\
10 & 0.65 & 0.68 & 0.70 & 0.74 & 0.78 & 0.80 & 0.83 & 0.85 & 0.85 \\
\hline
\end{tabular}
\caption{The sample means of the parameter estimates $\hat{\varphi} =
\hat{\phi}$ for ARMA$(1,2)$ processes with different parameter values
$\varphi$ using lags $N=1, 2, 3, \ldots, 10$. The MA parameters $\theta_1
= 0.8$ and $\theta_2 = 0.3$, the sample size is 50 and the number of
iterations is 1000}\vspace*{48pt}
\end{table}

\begin{table}[ht!]
%
\begin{tabular}{rrrrrrrrrr}
\hline
$N/\varphi$& 0.1 & 0.2 & 0.3 & 0.4 & 0.5 & 0.6 & 0.7 & 0.8 & 0.9 \\
\hline
1 & 0.09 & 0.19 & 0.29 & 0.39 & 0.49 & 0.59 & 0.69 & 0.80 & 0.90 \\
2 & 0.09 & 0.19 & 0.28 & 0.39 & 0.49 & 0.59 & 0.69 & 0.79 & 0.90 \\
3 & 0.12 & 0.18 & 0.26 & 0.37 & 0.48 & 0.59 & 0.69 & 0.79 & 0.90 \\
4 & 0.58 & 0.49 & 0.38 & 0.37 & 0.45 & 0.57 & 0.69 & 0.79 & 0.90 \\
5 & 0.64 & 0.65 & 0.62 & 0.57 & 0.52 & 0.56 & 0.68 & 0.79 & 0.90 \\
6 & 0.65 & 0.68 & 0.67 & 0.68 & 0.66 & 0.61 & 0.67 & 0.79 & 0.90 \\
7 & 0.66 & 0.68 & 0.69 & 0.71 & 0.72 & 0.69 & 0.69 & 0.78 & 0.90 \\
8 & 0.66 & 0.68 & 0.71 & 0.72 & 0.75 & 0.73 & 0.72 & 0.78 & 0.90 \\
9 & 0.66 & 0.68 & 0.71 & 0.74 & 0.76 & 0.75 & 0.76 & 0.77 & 0.90 \\
10 & 0.65 & 0.68 & 0.71 & 0.74 & 0.77 & 0.78 & 0.78 & 0.78 & 0.89 \\
\hline
\end{tabular}
\caption{The sample means of the parameter estimates $\hat{\varphi} =
\hat{\phi}$ for ARMA$(1,2)$ processes with different parameter values
$\varphi$ using lags $N=1, 2, 3, \ldots, 10$. The MA parameters $\theta_1
= 0.8$ and $\theta_2 = 0.3$, the sample size is 500 and the number of
iterations is 1000}\vspace*{-9pt}
\end{table}

\begin{table}[ht!]
%
\begin{tabular}{rrrrrrrrrr}
\hline
$N/\varphi$ & 0.1 & 0.2 & 0.3 & 0.4 & 0.5 & 0.6 & 0.7 & 0.8 & 0.9 \\
\hline
1 & 0.10 & 0.20 & 0.30 & 0.40 & 0.50 & 0.60 & 0.70 & 0.80 & 0.90 \\
2 & 0.10 & 0.20 & 0.30 & 0.40 & 0.50 & 0.60 & 0.70 & 0.80 & 0.90 \\
3 & 0.10 & 0.19 & 0.30 & 0.40 & 0.50 & 0.60 & 0.70 & 0.80 & 0.90 \\
4 & 0.34 & 0.21 & 0.27 & 0.39 & 0.50 & 0.60 & 0.70 & 0.80 & 0.90 \\
5 & 0.61 & 0.55 & 0.40 & 0.37 & 0.48 & 0.60 & 0.70 & 0.80 & 0.90 \\
6 & 0.65 & 0.65 & 0.62 & 0.50 & 0.48 & 0.59 & 0.70 & 0.80 & 0.90 \\
7 & 0.63 & 0.68 & 0.67 & 0.63 & 0.58 & 0.57 & 0.70 & 0.80 & 0.90 \\
8 & 0.64 & 0.68 & 0.68 & 0.71 & 0.69 & 0.60 & 0.69 & 0.80 & 0.90 \\
9 & 0.65 & 0.69 & 0.69 & 0.73 & 0.71 & 0.67 & 0.68 & 0.80 & 0.90 \\
10 & 0.64 & 0.67 & 0.71 & 0.75 & 0.74 & 0.74 & 0.69 & 0.80 & 0.90 \\
\hline
\end{tabular}
\caption{The sample means of the parameter estimates $\hat{\varphi} =
\hat{\phi}$ for ARMA$(1,2)$ processes with different parameter values
$\varphi$ using lags $N=1, 2, 3, \ldots, 10$. The MA parameters $\theta_1
= 0.8$ and $\theta_2 = 0.3$, the sample size is 5000 and the number of
iterations is 1000}\vspace*{-9pt}
\end{table}

\begin{table}[ht!]
%
\begin{tabular}{rrrrrrrrrr}
\hline
$N/\varphi$& 0.1 & 0.2 & 0.3 & 0.4 & 0.5 & 0.6 & 0.7 & 0.8 & 0.9 \\
\hline
1 & 0.10 & 0.20 & 0.30 & 0.40 & 0.50 & 0.60 & 0.70 & 0.80 & 0.90 \\
2 & 0.10 & 0.20 & 0.30 & 0.40 & 0.50 & 0.60 & 0.70 & 0.80 & 0.90 \\
3 & 0.10 & 0.20 & 0.30 & 0.40 & 0.50 & 0.60 & 0.70 & 0.80 & 0.90 \\
4 & 0.13 & 0.19 & 0.30 & 0.40 & 0.50 & 0.60 & 0.70 & 0.80 & 0.90 \\
5 & 0.56 & 0.30 & 0.28 & 0.39 & 0.50 & 0.60 & 0.70 & 0.80 & 0.90 \\
6 & 0.62 & 0.60 & 0.41 & 0.37 & 0.50 & 0.60 & 0.70 & 0.80 & 0.90 \\
7 & 0.63 & 0.65 & 0.63 & 0.46 & 0.47 & 0.60 & 0.70 & 0.80 & 0.90 \\
8 & 0.64 & 0.66 & 0.68 & 0.63 & 0.49 & 0.59 & 0.70 & 0.80 & 0.90 \\
9 & 0.62 & 0.67 & 0.69 & 0.71 & 0.60 & 0.58 & 0.70 & 0.80 & 0.90 \\
10 & 0.65 & 0.67 & 0.71 & 0.73 & 0.70 & 0.59 & 0.70 & 0.80 & 0.90 \\
\hline
\end{tabular}
\caption{The sample means of the parameter estimates $\hat{\varphi} =
\hat{\phi}$ for ARMA$(1,2)$ processes with different parameter values
$\varphi$ using lags $N=1, 2, 3, \ldots, 10$. The MA parameters $\theta_1
= 0.8$ and $\theta_2 = 0.3$, the sample size is 50000 and the number of
iterations is 1000}\vspace*{-9pt}
\end{table}

\begin{table}[ht!]
%
\begin{tabular}{lccccccc}
\hline
$T$ & max & min & mean & median & sd & mad & skewness \\
\hline
50 &1.00 & 0.00 & 0.413 & 0.409 & 0.326 & 0.436 & 0.222 \\
500 & 0.999 & 0.00 & 0.502 & 0.495 & 0.218 & 0.187 & 0.207 \\
5000 & 0.726 & 0.319 & 0.501 & 0.497 & 0.058 & 0.056 & 0.456 \\
50000 & 0.561 & 0.443 & 0.501 & 0.502 & 0.019 & 0.019 & -0.058 \\
\hline
\end{tabular}
\caption{The effect of the sample size $T$ on the estimates $\hat
{\varphi} = \hat{\phi}$ for an AR$(1)$ process. The true parameter
value $\varphi= 0.5$ and the lag $N = 3$. The number of iterations is 1000}\vspace*{-24pt}
\end{table}

\begin{table}[ht!]
%
\begin{tabular}{lccccccc}
\hline
$T$ & max & min & mean & median & sd & mad & skewness \\
\hline
50 &0.999 & 0.00 & 0.399 & 0.425 & 0.250 & 0.264 & -0.062 \\
500 & 0.681 & 0.00 & 0.481 & 0.491 & 0.086 & 0.078 & -1.020 \\
5000 & 0.570 & 0.395 & 0.499 & 0.500 & 0.024 & 0.023 & -0.201 \\
50000 & 0.527 & 0.474 & 0.500 & 0.500 & 0.008 & 0.007 & 0.036 \\
\hline
\end{tabular}
\caption{The effect of the sample size $T$ on the estimates $\hat
{\varphi} = \hat{\phi}$ for an ARMA$(1,2)$ process. The MA parameters
$\theta_1 = 0.8$ and $\theta_2 = 0.3$, the true parameter value
$\varphi= 0.5$ and the lag $N = 3$. The number of iterations is 1000}
\end{table}

\begin{table}[ht!]
%
\begin{tabular}{lccccccc}
\hline
$\varphi$ & max & min & mean & median & sd & mad & skewness \\
\hline
0.1 & 1.00 & 0.00 & 0.257 & 0.097 & 0.322 & 0.144 & 1.056 \\
0.4 & 0.989 & 0.111 & 0.396 & 0.395 & 0.096 & 0.083 & 0.500 \\
0.6 & 0.738 & 0.476 & 0.602 & 0.601 & 0.041 & 0.043 & 0.211 \\
0.9 & 1.00 & 0.852 & 0.901 & 0.899 & 0.020 & 0.018 & 0.943 \\
\hline
\end{tabular}
\caption{The effect of the true parameter value $\varphi$ on the
estimates $\hat{\varphi} = \hat{\phi}$ for AR$(1)$ processes. The
sample size $T=5000$ and the lag $N = 3$. The number of iterations is 1000}
\end{table}

\begin{table}[ht!]
%
\begin{tabular}{lccccccc}
\hline
$\varphi$ & max & min & mean & median & sd & mad & skewness \\
\hline
0.1 & 0.273 & 0.00 & 0.096 & 0.098 & 0.067 & 0.082 & 0.144 \\
0.4 & 0.496 & 0.254 & 0.396 & 0.397 & 0.032 & 0.032 & -0.198 \\
0.6 & 0.650 & 0.540 & 0.600 & 0.600 & 0.018 & 0.019 & -0.061 \\
0.9 & 0.929 & 0.868 & 0.899 & 0.899 & 0.009 & 0.009 & -0.076 \\
\hline
\end{tabular}
\caption{The effect of the true parameter value $\varphi$ on the
estimates $\hat{\varphi} = \hat{\phi}$ for ARMA$(1,2)$ processes. The
MA parameters $\theta_1 = 0.8$ and $\theta_2 = 0.3$, the sample size
$T=5000$ and the lag $N = 3$. The number of iterations is 1000}
\end{table}

\begin{table}[ht!]
%
\begin{tabular}{lccccccc}
\hline
$N$ & max & min & mean & median & sd & mad & skewness \\
\hline
1 & 0.550 & 0.457 & 0.501 & 0.501 & 0.014 & 0.015 & 0.017 \\
3 & 0.726 & 0.319 & 0.501 & 0.497 & 0.058 & 0.056 & 0.456 \\
5 & 1.00 & 0.00 & 0.513 & 0.493 & 0.246 & 0.226 & 0.098 \\
7 & 1.00 & 0.00 & 0.525 & 0.558 & 0.326 & 0.395 & -0.216\\
\hline
\end{tabular}
\caption{The effect of the lag $N$ on the estimates $\hat{\varphi} =
\hat{\phi}$ for an AR$(1)$ process. The sample size $T=5000$ and the
true parameter value $\varphi= 0.5$. The number of iterations is 1000}
\end{table}

\begin{table}[ht!]
%
\begin{tabular}{lccccccc}
\hline
$N$ & max & min & mean & median & sd & mad & skewness \\
\hline
1 & 0.548 & 0.455 & 0.500 & 0.500 & 0.015 & 0.016 & 0.134 \\
3 & 0.570 & 0.395 & 0.499 & 0.500 & 0.024 & 0.023 & -0.201 \\
5 & 0.710 & 0.00 & 0.482 & 0.499 & 0.112 & 0.092 & -1.456 \\
7 & 1.00 & 0.00 & 0.576 & 0.613 & 0.282 & 0.275 & -0.488\\
\hline
\end{tabular}
\caption{The effect of the lag $N$ on the estimates $\hat{\varphi} =
\hat{\phi}$ for an ARMA$(1,2)$ process. The MA parameters $\theta_1 =
0.8$ and $\theta_2 = 0.3$, the sample size $T=5000$ and the true
parameter value $\varphi= 0.5$. The number of iterations is 1000}
\end{table}
\end{appendix}

\end{document}